\documentclass{birkjour_t2}
 \theoremstyle{definition}
 
 \theoremstyle{remark}

 \numberwithin{equation}{section}

\usepackage{graphicx}
\usepackage{amssymb, amsmath, amsthm}
\usepackage{amsfonts}
\usepackage{amssymb}
\usepackage{float}
\usepackage{mathrsfs}
\usepackage{enumitem} 
\newtheorem{theorem}{Theorem}

\newtheorem{lemma}[theorem]{Lemma}

\newtheorem{proposition}[theorem]{Proposition}
\newtheorem{remark}[theorem]{Remark}

\newcommand{\dis}{\displaystyle}

\newcommand{\R}{\mathbb R}

\renewcommand{\theequation}{\thesection.\arabic{equation}}
\numberwithin{equation}{section}
\numberwithin{theorem}{section}

\numberwithin{figure}{section}

\begin{document}

%
%
%
%
%
%
%
%
%

\title[H\"{o}lder continuity of solutions to the kinematic dynamo equations]
 {H\"{o}lder continuity of solutions to the kinematic dynamo equations}
 
\author{Susan Friedlander}

\address{Department of Mathematics\\
University of Southern California}

\email{susanfri@usc.edu}

\author{Anthony Suen}

\address{Department of Mathematics\\
University of Southern California}

\email{asuen@usc.edu}
 
\date{March 30, 2014}

\keywords{regularity of solutions to parabolic equations, kinematic dynamo equations}

\subjclass{35K10, 35B65, 76W05} 

\begin{abstract}
We study the propagation of regularity of solutions to a three dimensional system of linear parabolic PDE known as the kinematic dynamo equations. The divergence free drift velocity is assumed to be at the critical regularity level with respect to the natural scaling of the equations.
\end{abstract}

\maketitle
\section{Introduction}

Dynamo action is the mechanism by which a magnetic field is generated and sustained through the motion of an electrically conducting fluid. One of the goals of dynamo theory is to understand the processes of the magnetic field in certain planets and stars. There is a considerable body of work on this topic in the astrophysical, geophysical and mathematical literature including the classical book of Moffatt \cite{Mof78}. In this context the full coupled system of equations of magnetohydrodynamics is extremely challenging and at present inaccessible, even numerically, in parameter ranges relevant to the geodynamo or the solar dynamo. One approach to simplify the system is to study the so called kinematic dynamo problem. In this approach the Lorentz force is neglected in the evolution equations governing the velocity of a convective, rotating, self gravitationally fluid. The magnetic induction equation then becomes a linear equation with the velocity being externally prescribed from the dynamically driven fluid flow. In reality, such a flow in a planetary interior is expected to be highly nontrivial with convective plumes and cells associated with chaotic particle paths. Over the past decades various models for kinematic dynamos have been proposed and studied: for example, Childress and Soward \cite{CS72}, Soward and Childress \cite{SC86}, Jones and Roberts \cite{JR00}. The importance of exponential stretching of the fluid trajectories for dynamo action was observed by Arnold \cite{A82}, Bayly \cite{B86}, Vishik \cite{V89} and Friedlander and Vishik \cite{FV91}. We note that flows with exponential stretching possess strong stochastic properties. More recent numerical treatments of the kinematic dynamo equations include those of Ponty et al \cite{PGS01}, Favier and Proctor \cite{FP13}. Even though the preponderance of results in this area have been restricted to the case of steady fluid flows there are some interesting articles where the velocities are taken to be non-stationary, for example Molchanov et al \cite{MRS85} and Finn and Ott \cite{FO88} who related the dynamo growth rate with quantitative measures of chaos in the flow. We note that although the kinematic dynamo equations are more tractable when the prescribed velocity is taken to be steady, a more realistic picture of a planetary interior would be a highly complex, unsteady flow. With this in mind, in our present paper we examine certain regularity properties of the kinematic dynamo equation under fairly minimal conditions on a prescribed unsteady velocity field. We do not address the question of dynamo instabilities. Rather we prove that any initial regularity of the magnetic field is guaranteed to persist even when it is deformed by non stationary flows of rather ``wild'' behavior. 

We are interested the following system of parabolic equations on $\R^3\times[0,\infty)$:

\begin{align}
\label{1.1} \left\{ \begin{array}{l}
B_t+u\cdot\nabla B=B\cdot\nabla u+\eta\Delta B, \\
\nabla\cdot B=0\\
\end{array}\right.
\end{align}
subjected to the initial condition
\begin{align}\label{1.2}
B(x,0)=B_0(x).
\end{align}
Here $x\in\R^3$ is the spatial coordinate, $t\ge0$ is the time variable and $\eta$ is a positive constant. $B=(B_1(x,t),B_2(x,t),B_3(x,t))$ is an unknown vector-valued function in $\R^3$ and $u=(u_1(x,t),u_2(x,t),u_3(x,t))$ is a prescribed divergence-free vector field. In the context of kinematic dynamo theory, system \eqref{1.1} represents Faraday's law of induction equation in which $u$ is a prescribed velocity vector field and $B$ is the magnetic field. The main goal of the present work is to study the regularity of solutions to \eqref{1.1}-\eqref{1.2} when $u$ is given. More precisely, we observe the persistence of H\"{o}lder continuity for weak solutions to \eqref{1.1}-\eqref{1.2} under minimal regularity assumptions on $u$. 

We recall some results in the literature related to our work. The regularity of {\it scalar} solutions $\theta(x,t)$ to the drift diffusion equation
\begin{align}\label{1.3}
\theta_t+(u\cdot\nabla)\theta=\Delta\theta
\end{align}
is a classical problem in parabolic PDE's: see, for example Lady$\breve{z}$henskaya et al \cite{LSU68}. However proving H\"{o}lder regularity of weak solutions of \eqref{1.3} with divergence free drift as rough as $u\in L^{\infty}_t BMO_{x}^{-1}$ was proved only recently in Friedlander and Vicol \cite{FV11a} and, at the same time, by Seregin et al \cite{SSSZ12}. The issue of regularity for genuinely vector systems such as \eqref{1.1} is even more delicate and rather little is known in norms that are critical with respect to the inherent scaling of the equation. To our knowledge the best result to date is that of Silvestre and Vicol \cite{SV12} who consider the somewhat artificial vector system based on \eqref{1.3} with coupling of the components via a pressure gradient, namely
\begin{align}\label{1.4}
v_t+(u\cdot\nabla)v=\Delta v+\nabla P,\qquad\nabla\cdot v=0.
\end{align}
Their result is a propagation of regularity instead of a regularization result. They prove that if a certain scale invariant norm of the divergence free drift velocity $u$ is bounded, then for $\alpha\in(0,1)$ the $C^{\alpha}$ norm of $v$ at time $t$ is bounded in terms of its norm at time zero. The proof uses Campanato's characterization of the H\"{o}lder modulus of continuity and a maximum principle type argument.

Our present analysis of the kinematic dynamo equation \eqref{1.1} closely follows the approach in Silvestre and Vicol \cite{SV12}. The novelty of our analysis is that we are required to control the ``stretching'' term $B\cdot\nabla u$ which involves a derivative of $u$ as opposed to the pressure gradient in \eqref{1.4}. We illustrate the comment in \cite{SV12} that their method could be applied to other evolution systems; in our case to a model for the solar dynamo.

\medskip

The following is the main result of our work:

\begin{theorem} Let $\beta\in(0,1]$, $T>0$, and define a pair $(p,q)$ by $p=\frac{2}{\beta}$ and $q=\frac{3}{1-\beta}$. Assume that $u:\R^3\times[0,T]\rightarrow\R^3$ is a divergence-free vector field such that
\begin{align}\label{c1}
u\in L^p([0,T];W^{1,q}(\R^3)).
\end{align}
Given $\alpha\in(0,1)$, assume also that $B_0$ satisfies $B_0\in L^2(\R^3)\cap C^{\alpha}(\R^3)$ with $\nabla\cdot B_0=0$. Then there exists a weak solution $B:\R^3\times[0,T]\rightarrow\R^3$ of \eqref{1.1}-\eqref{1.2} such that $B(x,t)$ is $C^{\alpha}$ in $x$ for all $t\in[0,T]$ and we have
\begin{align*}
[B(\cdot,t)]_{C^{\alpha}(\R^3)}\le C[B_0]_{C^{\alpha}(\R^3)},
\end{align*}
for some positive universal constant $C$ depending on $\eta,T,\alpha,\beta$, $\|B_0\|_{L^2(\R^3)}$ and $\|u\|_{L^p([0,T];W^{1,q}(\R^3)}$. 
\end{theorem}


\begin{remark}
We point out that the drift velocity $u$ in \eqref{1.1} is assumed to be {\it critical} with respect to the natural scaling of the system of equations. In other words, the norm $\|u\|_{L^p([0,T];W^{1,q}(\R^3)}$ remains invariant under the transformation $u(x,t)\mapsto ru(rx,r^2t)$ for any $r>0$. We refer the readers to \cite{SV12} for a more detailed discussion about critical assumptions on the drift velocity $u$.
\end{remark}

\begin{remark}
We also recall the result obtained by Lady$\breve{z}$henskaya et al in \cite{LSU68} which says that under the assumptions $u,\nabla u\in L^p_t L^q_x$ with $\frac{2}{p}+\frac{3}{q}\le1$ and the initial condition $B_0\in L^2(\R^3)$, then the system \eqref{1.1} has a unique weak solution. Our present work addresses the propagation of regularity which is different from the existence theory shown in \cite{LSU68}, yet part of our analysis is reminiscent of those given therein.
\end{remark}

\begin{remark}
On replacing $B$ by curl $u$, equation \eqref{1.1} becomes the three dimensional Navier-Stokes equations in vorticity form. In this context Theorem~1.1 gives a no blow-up criterion for the Navier-Stokes equations. It says that if the norm of $u$ remains bounded in $L^p_t W^{1,q}_x$ with $\nabla\times u_0\in L^2_x$, then $\nabla\times u$ does not blow up on $\R^3\times[0,T]$. This result is closely connected to the scale invariant no blow-up condition for the velocity observed in Silvestre and Vicol \cite{SV12} and the classical Lady$\breve{z}$henskaya-Foia$\text{\c{s}}$-Prodi-Serrin condition. 
\end{remark}


\section{Preliminaries}

We introduce the following notations. $\mathcal B_r(x)$ is the ball of radius $r$ centered at $x$, and $\phi$ is defined to be a smooth function supported on $\mathcal B_1(0)$ with integral value 1. The weighted mean of $B$ on $\mathcal B_r(x)$ is defined by
\begin{align}\label{d1}
\bar B(x,t,r)=\int_{\mathcal B_1(0)}B(x+ry,t)\phi(y)dy.
\end{align}

The integral version of the modulus of continuity of $B$ is given by
\begin{align}\label{d2}
\mathcal I(x,t,r)=\int_{\mathcal B_1(0)}|B(x+ry,t)-\bar B(x,t,r)|^2\phi(y)dy.
\end{align}







We have the following proposition which is reminiscent of the one obtained by Slivestre and Vicol cf. [\cite{SV12}, Proposition~2.2]:

\begin{proposition}
Let $B^\varepsilon$ and $u^\varepsilon$ be a sequence of smooth divergence-free vector field. Assume that $B^\varepsilon$ is a weak
solution of \eqref{1.1} with drift velocity $u^\varepsilon$ as shown by Lady$\breve{z}$henskaya et al cf. [\cite{LSU68}, Theorem~1.1 in Chapter VII]. Assume also that $u^\varepsilon,\nabla u^\varepsilon\rightarrow u,\nabla u$ strongly in $L^1
_{loc}L^1_{loc}$. Then, up to a subsequence,
$B^\varepsilon$ converges weakly to a weak solution of \eqref{1.1}.
\end{proposition}

We recall the following proposition which is useful in proving Theorem~1.1. The proof can be found in \cite{C63}.

\begin{proposition}
(Campanato's characterization of H\"{o}lder spaces) Let $f:\R^n\rightarrow\R^m$ be an $L^2$ function such that for all $r>0$ and $x\in\R^n$, there exists a constant $\bar f$ such that
\begin{align*}
\int_{\mathcal B_1(0)}|f(x+ry)-\bar f|\phi(y)dy\le A^2r^{2\alpha}
\end{align*}
for some positive constant $A$, and $\alpha\in(0,1)$. Then the function $f$ has a H\"{o}lder continuous representative such that 
\begin{align*}
|f(x)-f(y)|\le \mathcal KA|x-y|^\alpha
\end{align*}
where the constant $\mathcal K>0$ depends on dimensions and $\alpha$ only.
\end{proposition}
We also recall the following Sobolev inequality which is standard cf. [Ziemer \cite{Z89}, Theorem~2.4.1]:

\begin{proposition}
(Sobolev inequality in $\R^3$) For any $d\in[2,6]$, there exists $M(d)>0$ such that for each $\psi\in W^{1,2}(\R^3)$,
\begin{align}\label{si1}
\int_{\R^3}|\psi|^d\le M(d)\left(\int_{\R^3}|\psi|^2\right)^\frac{6-d}{4}\left(\int_{\R^3}|\nabla\psi|^2\right)^\frac{3d-6}{4}.
\end{align}
\end{proposition}

\section{Proof of the main result}

We begin our proof of Theorem~1.1 with the following auxiliary lemma about the drift velocity $u$. 

\begin{lemma}
Let $\beta\in(0,1]$, $p,q$ and $T>0$ be in Theorem~1.1. Assume that $u$ satisfies \eqref{c1} with $\|u(\cdot,t)\|_{W^{1,q}(\R^3)}\le g(t)$ for some $g(t)\in L^p([0,T])$. Then for all $t\in(0,T]$, we have
\begin{align}
[u(\cdot,t)]_{C^{\beta}(\R^3)}&\le k_1 g(t),\label{1.2-1}\\
\sup_{x\in\R^3}\sup_{r>0} r^{-\beta}\int_{\mathcal B_1(0)}|u(x+ry,t)-u(x,t)|dy&\le k_1 g(t),\label{1.2-2}\\
\sup_{x\in\R^3}\sup_{r>0} r^{-\beta+1}\int_{\mathcal B_1(0)}|\nabla u(x+ry,t)|dy&\le k_2 g(t).\label{1.2-3}
\end{align}
Here $k_1,k_2$ are some positive constants.
\end{lemma}

\begin{proof}
We first consider the case for $\beta\in(0,1)$. To prove \eqref{1.2-1}, we notice that by the definition of $q$ we have $\beta=1-\frac{3}{q}$. So by Morrey's inequality, there exists $k_1>0$ such that
\begin{align*}
[u(\cdot,t)]_{C^{\beta}(\R^3)}\le k_1\|u(\cdot,t)\|_{W^{1,q}(\R^3)}\le k_1 g(t).
\end{align*}
Next we prove \eqref{1.2-2} and \eqref{1.2-3}. The assertion \eqref{1.2-2} follows from \eqref{1.2-1} and the equivalence between Morrey-Campanato space and H\"{o}lder class cf. \cite{SV12}. To show \eqref{1.2-3}, we fix $x\in\R^3$, $r>0$ and $t\in(0,T]$. Using H\"{o}lder's inequality, there exists $k_2>0$ such that
\begin{align*}
r^{-\beta+1}\int_{\mathcal B_1(0)}|\nabla u(x+ry,t)|dy&\le \frac{k_2r^{-\beta+1}}{r^3}\int_{\mathcal B_r(x)}|\nabla u(z)|dz\\
&\le k_2 r^{-{\beta-2}}\left[\int_{\mathcal B_r(x)}|\nabla u(z)|^q dz\right]^\frac{1}{q}\left[\int_{\mathcal B_r(x)}1 dz\right]^{1-\frac{1}{q}}\\
&\le k_2 r^{-{\beta-2}}\|u(\cdot,t)\|_{W^{1,q}(\R^3)}(r^3)^{1-\frac{1}{q}}\\
&\le k_2 r^{-{\beta-2}}g(t)r^{\beta+2}=k_2 g(t).
\end{align*}
Hence by taking supremum over $x$ and $r$, \eqref{1.2-3} follows immediately.

For the case when $\beta=1$, first we note that $u(\cdot,t)\in W^{1,\infty}(\R^3)$ if and only if $u$ is Lipschitz continuous with $\|u(\cdot,t)\|_{C^{1}(\R^3)}\le k_3\| u(\cdot,t)\|_{W^{1,\infty}(\R^3)}$ for some constant $k_3$, and hence \eqref{1.2-1} follows. Then \eqref{1.2-2} and \eqref{1.2-3} are direct consequences of \eqref{1.2-1} and we omit the details here. 
\end{proof}

Next we prove the following $L^2$ estimates on $B$ in terms of $u$.

\begin{lemma}
Assume that the notations as in Theorem~1.1 are in force, and let $u:\R^3\times[0,T]\rightarrow\R^3$ be a divergence-free vector field such that
\begin{align*}
u\in L^p([0,T];L^{q}(\R^3)).
\end{align*}
Then there exists $c>0$ such that,
\begin{align}\label{1.3-0}
\sup_{0\le t\le T}\int_{R^3}|B(x,t)|^2dx+\int_0^T\!\!\!\int_{R^3}|\nabla B|^2\le c\|B_0\|^2_{L^2(\R^3)},
\end{align}
where $c$ depends on $T$ and $\|u\|_{L^p([0,T],L^q(\R^3))}$.
\end{lemma}

\begin{remark}
We point out that the assumption on $u$ in Lemma~3.2 is weaker than \eqref{c1} in Theorem~1.1.
\end{remark}

\begin{proof}
For any $0\le s\le t\le T$, we define 
$$\dis\Phi(s,t)=\sup_{s\le\tau\le t}\int_{R^3}|B(x,\tau)|^2dx+\int_s^t\!\!\!\int_{R^3}|\nabla B|^2.$$ 
We multiply the first equation of \eqref{1.1} by $B$ and integrate to obtain, for any $\tau\in[s,t]$,
\begin{align}\label{1.3-1}
\frac{\eta}{2}\int_s^{\tau}\!\!\!\int_{\R^3}|\nabla B|^2+\frac{1}{2}\int_{\R^3}|B(x,\tau)|^2-\frac{1}{2}\int_{\R^3}|B(x,s)|^2=\int_s^{\tau}\!\!\!\int_{\R^3}[(B\cdot\nabla)u]\cdot B.
\end{align}
Recall the definitions for the pair $(p,q)$ which are given by $p=\frac{2}{\beta}$ and $q=\frac{3}{1-\beta}$ with $\beta\in(0,1]$. Using integration by parts, H\"{o}lder's inequality and the Sobolev inequality \eqref{si1}, there exists $M=M(d)>0$ with $\frac{1}{q}+\frac{1}{d}+\frac{1}{2}=1$ and $d=\frac{6}{1+2\beta}\in[2,6]$ such that the right hand side of \eqref{1.3-1} is bounded by
\begin{align*}
&\left|\int_s^{\tau}\!\!\!\int_{\R^3}[(B\cdot\nabla)u]\cdot B\right|\\
&\le\int_s^t\!\!\!\int_{\R^3}|B||\nabla B||u|\\
&\le\int_s^t\left(\int_{\R^3}|u|^q\right)^\frac{1}{q}\left(\int_{\R^3}|B|^d\right)^\frac{1}{d}\left(\int_{\R^3}|\nabla B|^2\right)^\frac{1}{2}\\
&\le \int_s^t\left(\int_{\R^3}|u|^q\right)^\frac{1}{q}\left[M\left(\int_{\R^3}|B|^2\right)^\frac{6-d}{4d}\left(\int_{\R^3}|\nabla B|^2\right)^\frac{3d-6}{4d}\right]\left(\int_{\R^3}|\nabla B|^2\right)^\frac{1}{2}\\
&\le M\|u\|_{L^p([s,t],L^q(\R^3))}\left(\sup_{s\le\tau\le t}\int_{\R^3}|B(x,\tau)|^2dx\right)^\frac{6-d}{4d}\left[\int_s^{\tau}\left(\int_{\R^3}|\nabla B|^2\right)^{\frac{5d-6}{4d}\cdot\frac{p}{p-1}}\right]^{\frac{p-1}{p}}\\
&= M\|u\|_{L^p([s,t],L^q(\R^3))}\left(\sup_{s\le\tau\le t}\int_{\R^3}|B(x,\tau)|^2dx\right)^\frac{6-d}{4d}\left(\int_s^{\tau}\!\!\!\int_{\R^3}|\nabla B|^2\right)^{\frac{p-1}{p}}\\
&\le M\Phi(s,t)\|u\|_{L^p([s,t],L^q(\R^3))},
\end{align*}
which follows from the facts that $\frac{5d-6}{4d}\cdot\frac{p}{p-1}=\frac{4-2\beta}{4}\cdot\frac{2}{2-\beta}=1$ and $\frac{6-d}{4d}+\frac{p-1}{p}=\frac{\beta}{2}+\frac{2-\beta}{2}=1$. Define $\delta=\frac{1}{M}\min\{\frac{1}{2},\frac{\eta}{2}\}$. By taking supremum over $\tau\in[s,t]$, we obtain
\begin{align}\label{1.3-2}
M\delta\Phi(s,t)\le M\Phi(s,t)\|u\|_{L^p([s,t],L^q(\R^3))}+\frac{1}{2}\|B(\cdot,s)\|^2_{L^2(\R^3)}.
\end{align}
Following the same idea in \cite{LSU68}, we subdivide $[0,T]$ into a finite number of intervals $[0,t_1], [t_1,t_2], \dots,$ $[t_{m-1},t_{m}=T]$ such that  for each $k\in\{1,2,\dots,m\}$,
\begin{align*}
\frac{\delta}{4}\le\|u\|_{L^p([t_{k-1},t_k],L^q(\R^3))}\le \frac{\delta}{2}
\end{align*}
with $m\le 1+\frac{4}{\delta}\|u\|_{L^p([0,T],L^q(\R^3))}$. Hence for each $k$, we have from \eqref{1.3-2} that
\begin{align}\label{1.3-3}
\frac{M\delta}{2}\Phi(t_{k-1},t_k)\le\frac{1}{2}\|B(\cdot,t_{k-1})\|^2_{L^2(\R^3)}
\end{align}
We iterate \eqref{1.3-3} and conclude
\begin{align*}
\Phi(0,T)\le(\frac{1}{M\delta})^m\|B_0\|_{L^2(\R^3)}^2\le (\frac{1}{M\delta})^{1+\frac{4}{\delta}\|u\|_{L^p([0,T],L^q(\R^3))}}\|B_0\|_{L^2(\R^3)}^2.
\end{align*}
\end{proof}

\begin{proof}[Proof of Theorem~1.1]
Using Proposition~2.1, it is enough to prove Theorem~1.1 by assuming that $u$ is smooth and $B$ is a classical solution cf. [\cite{SV12}, Proposition~2.3].

Recall the modulus of continuity of $B$ given by $\mathcal I$ in \eqref{d2} that
\begin{align*}
\mathcal I(x,t,r)=\int_{\mathcal B_1(0)}|B(x+ry,t)-\bar B(x,t,r)|^2\phi(y)dy.
\end{align*}
To prove Theorem~1.1, it suffices to show that under the assumptions on $u$ as in Theorem~1.1, we have
\begin{align}\label{p0}
\mathcal I(x,t,r)< f(t)^2 r^{2\alpha}
\end{align}
for some $f(t)>0$ and for all $x, r$ and $t$. For the sake of contradiction, suppose there exists a first time $t>0$ and some values of $x,r$ such that
\begin{align*}
\mathcal I(x,t,r)=f(t)^2r^{2\alpha}.
\end{align*}
Assume that $\|u(\cdot,t)\|_{W^{1,q}(\R^3)}\le g(t)$ for some $g(t)\in L^p([0,T])$, and without loss of generality we may take $g(t)>0$. Because $\mathcal I<f(t)^2r^{2\alpha}$ for all times prior to $t$ and the assumption that $(B-\bar B)$ has zero mean, we obtain
\begin{align}\label{p1}
2f'(t)f(t)r^{2\alpha}&\le\partial_t\mathcal I\notag\\
&=\int_{\mathcal B_1(0)}(B(x+ry,t)-\bar B(x,t,r))\partial_t B(x+ry,t)\phi(y)dy\notag\\
&=\int_{\mathcal B_1(0)}(B(x+ry,t)-\bar B(x,t,r))\cdot[-u(x+ry,t)\cdot\nabla_x B(x+ry,t)\notag\\
&\qquad+\eta\Delta_x B(x+ry,t)+B(x+ry,t)\cdot\nabla_x u(x+ry,t)]\phi(y)dy\notag\\
&:=\mathcal A+\mathcal D+\mathcal E.
\end{align}
We first consider the terms $\mathcal A$ and $\mathcal D$. We claim that there exists $C_1,C_2>0$ such that
\begin{align}\label{p2}
\mathcal A&\le C_1 r^{2\alpha+\beta-1}f(t)^2 g(t),\\
\mathcal D&\le-C_2 r^{2\alpha-2}f(t)^2.\label{p3}
\end{align}
We only give a sketch of the proof of \eqref{p2}-\eqref{p3} since the details can be found \cite{SV12}. To show \eqref{p2}, we observe that since $\mathcal I$ achieves its maximum at $x$ (for fixed $t$ and $r$), and we have
\begin{align*}
0=\nabla_x\mathcal I=\int_{\mathcal B_1(0)}(B(x+ry,t)-\bar B(x,t,r))\cdot(\nabla_x B(x+ry,t)-\nabla_x\bar B(x,t,r))\phi(y)dy. 
\end{align*}
By the definition of $\bar B$ and the fact that $\nabla_x\bar B(x,t,r)$ is independent of $y$, 
\begin{align*}
0=\int_{\mathcal B_1(0)}(B(x+ry,t)-\bar B(x,t,r))\cdot(\nabla_x\bar B(x,t,r))\phi(y)dy
\end{align*}
and hence we obtain the following identity
\begin{align}\label{p7}
0=\int_{\mathcal B_1(0)}(B(x+ry,t)-\bar B(x,t,r))\cdot(\nabla_x B(x+ry,t))\phi(y)dy.
\end{align}
Upon integrating by parts, applying \eqref{p7} and using the fact that $\dis\nabla_x B(x+ry)=r^{-1}\nabla_y B(x+ry)$, we have
\begin{align*}
\mathcal A&=r^{-1}\int_{\mathcal B_1(0)}|B(x+ry,t)-\bar B(x,t,r)|^2 (u(x+ry,t)-u(x,t))\cdot\nabla\phi(y)dy.
\end{align*}
Using Proposition~2.2, $|B(x+ry,t)-\bar B(x,t,r)|^2$ is bounded by $\mathcal Kr^{2\alpha}f(t)^2$, and with the help of \eqref{1.2-2}, we have
\begin{align*}
\int_{\mathcal B_1(0)}|u(x+ry,t)-u(x,t)||\nabla\phi(y)|dy&\le \sup_{y\in \mathcal B_1(0)}\|\nabla\phi(y)\|_{L^\infty(\R^3)}k_1r^\beta g(t).
\end{align*}
Hence we obtain
\begin{align*}
\mathcal A\le \sup_{y\in \mathcal B_1(0)}\|\nabla\phi(y)\|_{L^\infty(\R^3)}\mathcal K k_1 r^{2\alpha+\beta-1}f(t)^2g(t)
\end{align*}
and \eqref{p2} follows. 

Next we show \eqref{p3}. First, we can rewrite the functional $\mathcal I$ into the form
\begin{align*}
\mathcal I(x,t,r)=\int\int_{\mathcal B_1(0)\times \mathcal B_1(0)}|B(x+ry,t)-B(x+rz,t)|^2\phi(y)\phi(z)dydz.
\end{align*}
Using the assumption that $\mathcal I(x,t,r)$ attains $f(t)^2r^{2\alpha}$ for the first time, we have $\nabla_x\mathcal I=0$ and $\Delta_x \mathcal I\le0$, and hence
\begin{align*}
0&\ge\Delta_x I(x,t,r)\\
&=2\mathcal D+2\int\int_{\mathcal B_1(0)\times \mathcal B_1(0)}|\nabla_x B(x+ry,t)-\nabla_x B(x+rz,t)|^2\phi(y)\phi(z)dydz,
\end{align*}
which implies
\begin{align}\label{p8}
\mathcal D&\le-\int\int_{\mathcal B_1(0)\times \mathcal B_1(0)}|\nabla_x B(x+ry,t)-\nabla_x B(x+rz,t)|^2\phi(y)\phi(x)dydz\\
&\le-r^{-2}\int\int_{\mathcal B_1(0)\times \mathcal B_1(0)}|\nabla_y B(x+ry,t)-\nabla_z B(x+rz,t)|^2\phi(y)\phi(z)dydz.\notag
\end{align}
We need to estimate the right side of \eqref{p8} in terms of $\mathcal I$ and $\mathcal I_r$. We notice that, by direct computation,
\begin{align}\label{p9}
2\mathcal I-r\partial_r \mathcal I&=\int\int_{\mathcal B_1(0)\times \mathcal B_1(0)}|B(x+ry,t)-B(x+rz,t)|^2\phi(y)\phi(z)dydz\notag\\
&\qquad-\int\int_{\mathcal B_1(0)\times \mathcal B_1(0)}(y\cdot\nabla_y B(x+ry,t)-z\cdot\nabla_z B(x+rz,t))\\
&\qquad\qquad\qquad\qquad\qquad\qquad\cdot(B(x+ry,t)-B(x+rz,t))\phi(y)\phi(z)dydz.\notag
\end{align}
The right hand side of \eqref{p9} can be estimated in the following way cf. [\cite{SV12} Lemma~3.5]:
\begin{align*}
&\int\int_{\mathcal B_1(0)\times \mathcal B_1(0)}|B(x+ry,t)-B(x+rz,t)|^2\phi(y)\phi(z)dydz\\
&-\int\int_{\mathcal B_1(0)\times \mathcal B_1(0)}(y\cdot\nabla_y B(x+ry,t)-z\cdot\nabla_z B(x+rz,t))\\
&\qquad\qquad\qquad\qquad\qquad\qquad\qquad\cdot(B(x+ry,t)-B(x+rz,t))\phi(y)\phi(z)dydz\\
&\le \mathcal C\left(\int\int_{\mathcal B_1(0)\times \mathcal B_1(0)}|\nabla_y B(x+ry,t)-\nabla_z B(x+rz,t)|^2\phi(y)\phi(z)\right)^\frac{1}{2}\\
&\qquad\qquad\times\left(\int\int_{\mathcal B_1(0)\times \mathcal B_1(0)}|B(x+ry,t)-B(x+rz,t)|^2\phi(y)\phi(z)\right)^\frac{1}{2}\\
&=\mathcal C\left(\int\int_{\mathcal B_1(0)\times \mathcal B_1(0)}|\nabla_y B(x+ry,t)-\nabla_z B(x+rz,t)|^2\phi(y)\phi(z)\right)^\frac{1}{2}\mathcal I^\frac{1}{2}
\end{align*}
for some constant $\mathcal C>0$, hence we obtain from \eqref{p9} that
\begin{align*}
-r^{-2}\int\int_{\mathcal B_1(0)\times \mathcal B_1(0)}|\nabla_y B(x+ry,t)-\nabla_z &B(x+rz,t)|^2\phi(y)\phi(z)dydz\\
&\le -\frac{r^{-2}}{\mathcal C^2}\frac{(2\mathcal I-r\partial_r\mathcal I)^2}{\mathcal I}.
\end{align*}
At the first point where $\mathcal I(x,t,r)=f(t)^2r^{2\alpha}$ holds, we should also have $\partial_r \mathcal I=2\alpha f(t)^2 r^{2\alpha-1}$. So we conclude
\begin{align}\label{p10}
-r^{-2}\int\int_{\mathcal B_1(0)\times \mathcal B_1(0)}|\nabla_y B(x+ry,t)-\nabla_z &B(x+rz,t)|^2\phi(y)\phi(z)dydz\notag\\
&\le-\frac{(1-\alpha)^2f(t)^2r^{2\alpha-2}}{\mathcal C^2}.
\end{align}
We apply \eqref{p10} on \eqref{p8}, and \eqref{p3} follows immediately.

Finally, we estimate $\mathcal E$ in \eqref{p1}. We note that the term $\mathcal E$ differs from the comparable term in \cite{SV12} and our estimate of $\mathcal E$ is the novelty of this current article. We subdivide $\mathcal E$ into $\mathcal E_1$ and $\mathcal E_2$:
\begin{align}\label{p4}
\mathcal E&=\int_{\mathcal B_1(0)}(B(x+ry,t)-\bar B(x,t,r))\cdot[B(x+ry,t)\cdot\nabla_x u(x+ry,t)]\phi(y)dy\notag\\
&=\int_{\mathcal B_1(0)}(B(x+ry,t)-\bar B(x,t,r))\cdot[\nabla_x u(x+ry,t)\cdot(B(x+ry,t)-\bar B(x,t,r))]\phi(y)dy\notag\\
&+\int_{\mathcal B_1(0)}(B(x+ry,t)-\bar B(x,t,r))\cdot[\nabla_x u(x+ry,t)\cdot\bar B(x,t,r))]\phi(y)dy\notag\\
&:=\mathcal E_1+\mathcal E_2.
\end{align}
Using \eqref{1.2-3}, the term $\mathcal E_1$ is bounded by
\begin{align}\label{p5}
\mathcal E_1&\le\left|\int_{\mathcal B_1(0)}(B(x+ry,t)-\bar B(x,t,r))\cdot[\nabla_x u(x+ry,t)\cdot(B(x+ry,t)-\bar B(x,t,r))]\phi(y)dy\right|\notag\\
&\le C_3f(t)^2r^{2\alpha}\int_{\mathcal B_1(0)}|\nabla_x u(x+ry,t)|\phi(y)dy\notag\\
&\le C_3f(t)^2r^{2\alpha}r^{2\alpha+\beta-1}g(t),
\end{align}
where $C_3>0$ is a constant. The term $\mathcal E_2$ can be estimated by
\begin{align*}
\mathcal E_2&\le\left|\int_{\mathcal B_1(0)}(B(x+ry,t)-\bar B(x,t,r))\cdot[\nabla_x u(x+ry,t)\cdot\bar B(x,t,r))]\phi(y)dy\right|\notag\\
&\le f(t)r^\alpha\left[\int_{\mathcal B_1(0)}|\nabla_x u(x+ry,t)|\phi(y)|dy\right]|\bar B(x,r,t)|\notag\\
&\le r^{\alpha+\beta-1}f(t)g(t)\left|\int_{\mathcal B_1(0)}B(x+rz,t)\phi(z)dz\right|\notag\\
&\le r^{\alpha+\beta-1}f(t)g(t)\left[\int_{\mathcal B_1(0)}|B(x+rz,t)-B(x+z,t)||\phi(z)|dz+\int_{\mathcal B_1(0)}|B(x+z,t)||\phi(z)|dz\right].
\end{align*} 
Using Proposition~2.2, the term $\dis\int_{\mathcal B_1(0)}|B(x+rz,t)-B(x+z,t)||\phi(z)|dz$ can be bounded by
\begin{align*}
\mathcal K\int_{\mathcal B_1(0)}f(t)|x+rz-xz|^\alpha|\phi(z)|dz\le \mathcal Kf(t)(r+1)^\alpha,
\end{align*}
while for the term $\dis\int_{\mathcal B_1(0)}|B(x+z,t)||\phi(z)|dz$, using \eqref{1.3-0} from Lemma~3.2, 
\begin{align*}
\int_{\mathcal B_1(0)}|B(x+z,t)||\phi(z)|dz&\le\|B(\cdot,t)\|_{L^2(\R^3)}\|\phi(\cdot)\|_{L^2(\mathcal B_1(0))}\\
&\le c\|B_0\|_{L^2(\R^3)}\|\phi(\cdot)\|_{L^2(\mathcal B_1(0))}.
\end{align*}
Hence we obtain
\begin{align}\label{p6}
\mathcal E_2\le C_4r^{\alpha+\beta-1}f(t)^2g(t)+C_5r^{2\alpha+\beta-1}f(t)^2g(t)+C_6r^{\alpha+\beta-1}f(t)g(t)
\end{align}
for some positive constants $C_4,C_5,C_6$. 

Putting the bounds \eqref{p2}, \eqref{p3}, \eqref{p5} and \eqref{p6} together,
\begin{align*}
f(t)'f(t)r^{2\alpha}&\le C_1 r^{2\alpha+\beta-1}f(t)^2 g(t)-C_2 r^{2\alpha-2}f(t)^2+C_3r^{2\alpha+\beta-1}f(t)^2g(t)\\
&+C_4r^{\alpha+\beta-1}f(t)^2g(t)+C_5r^{2\alpha+\beta-1}f(t)^2g(t)+C_6r^{\alpha+\beta-1}f(t)g(t).
\end{align*}
Hence there exists constants $M_1,M_2>0$ such that
\begin{align*}
f(t)^\frac{-2}{1+\beta-\alpha}\left[f(t)'f(t)^\frac{1-\beta+\alpha}{1+\beta-\alpha}-M_1g(t)^\frac{2}{1+\beta-\alpha}\right]\le M_2\left[g(t)^\frac{2}{1+\beta}+g(t)^\frac{2}{1+\beta-\alpha}\right].
\end{align*}
Define $G_1(t)=M_1g(t)^\frac{2}{1+\beta-\alpha}$, $G_2(t)=M_2\left[g(t)^\frac{2}{1+\beta}+g(t)^\frac{2}{1+\beta-\alpha}\right]$ and $F(t)=f(t)^\frac{2}{1+\beta-\alpha}$, by the assumption on $\alpha,\beta,g$ and the fact that $\max\{\frac{2}{1+\beta},\frac{2}{1+\beta-\alpha}\}<\frac{2}{\beta}=p$, we have $G_1,G_2\in L^1_t$ with $G_1,G_2>0$. We then obtain the following inequality
\begin{align}\label{ineq}
F(t)^{-1}\left(\frac{1+\beta-\alpha}{2}F(t)'-G_1(t)\right)\le G_2(t).
\end{align}
On the other hand, we can choose $F(t)$ to be the solution of the O.D.E.
\begin{align*}
\frac{1+\beta-\alpha}{2}F(t)'-G_1(t)=2F(t)G_2(t)
\end{align*}
given explicitly by
\begin{align*}
F(t)=\int_0^t \frac{2G_1(s)e^{\int_s^t \frac{4G_2(\tau)}{1+\beta-\alpha}d\tau}}{1+\beta-\alpha}ds+e^{\int_0^t \frac{4G_2(s)}{1+\beta-\alpha}ds}F(0),
\end{align*}
which contradicts the inequality \eqref{ineq}. Hence it is impossible to invalidate \eqref{p0} and we finish the proof of Theorem~1.1.
\end{proof}


\subsection*{Acknowledgment}
We thank Vlad Vicol for very helpful discussions. This work was partially supported by NSF grant DMS-1207780 (SF).

\end{document}